\documentclass[leqno,12pt]{article} 
\usepackage{amsmath, amssymb}
\usepackage{amsthm} 
\usepackage{pictexwd,dcpic}
\theoremstyle{plain} 
\newtheorem{theorem}{\indent\sc Theorem}[section]
\newtheorem{lemma}[theorem]{\indent\sc Lemma}
\newtheorem{corollary}[theorem]{\indent\sc Corollary}

\theoremstyle{definition} 

\newtheorem{remark}[theorem]{\indent\sc Remark}

\makeatletter
\def\address#1#2{\begingroup
\noindent\parbox[t]{7.8cm}{%
\small{\scshape\ignorespaces#1}\par\vskip1ex
\noindent\small{\itshape E-mail address}%
\/: #2\par\vskip4ex}\hfill%
\endgroup}%
\makeatother
\title{Completely normal elements in finite abelian extensions} 
\author{
\textsc{Ja Kyung Koo and Dong Hwa Shin} 
}
\date{} 
\begin{document}

\maketitle

 \footnote{
\thanks{
This research was partially supported by Basic Science Research
Program through the NRF of Korea funded by MEST (2011-0001184). The
second named author was partially supported by TJ Park Postdoctoral
Fellowship.} }

\begin{abstract}
We give a completely normal element in the maximal real subfield of
a cyclotomic field over the field of rational numbers, which is
different from \cite{Okada1}. This result is a consequence of the
criterion for a normal element developed in \cite{J-K-S}.
Furthermore, we find a completely normal element in certain
extension of modular function fields in terms of a quotient of the
modular discriminant function.
\end{abstract}

\section{Introduction}

Let $L/K$ be a finite Galois extension. By the normal basis theorem
\cite{Waerden} there exists an element $x\in L$ such that
$\{x^\gamma~|~ \gamma\in\mathrm{Gal}(L/K)\}$ is a $K$-basis of $L$,
a so-called \textit{normal basis} of $L/K$. Such an element $x$ is
said to be \textit{normal} in $L/K$. Moreover, if $x$ is normal in
$L/F$ for every intermediate field $F$, then $x$ is said to be
\textit{completely normal} in $L/K$.  The existence of a completely
normal element was first proved by Blessenohl and Johnsen
\cite{B-J}.
\par
Throughout this paper we let $\zeta_\ell=e^{2\pi i/\ell}$ be a
primitive $\ell$th root of unity for a positive integer $\ell$.
Furthermore, we let $\mathbb{Q}(\zeta_\ell)^+$ be the maximal real
subfield of the $\ell$th cyclotomic field $\mathbb{Q}(\zeta_\ell)$.
\par
Okada \cite{Okada1} proved that if $k$ and $\ell$ ($>2$) are
positive integers with $k$ odd and $T$ is a set of representatives
for which $(\mathbb{Z}/\ell\mathbb{Z})^\times=T\cup(-T)$, then the
numbers $(1/\pi^k)(d/dz)^k(\cot \pi z)|_{z=a/\ell}$ for $a\in T$
form a normal basis of $\mathbb{Q}(\zeta_\ell)^+/\mathbb{Q}$, which
generalized the works of Chowla \cite{Chowla} when $k=1$. He
utilized the partial fractional decomposition of $(d/dz)^k(\cot \pi
z)$ and the Frobenius determinant relation \cite[Chapter 21 Theorem
5]{Lang}.
\par
On the other hand, let $\mathcal{N}$ be the set of positive integers
which are either odd or divisible by $4$. Let $\ell\in\mathcal{N}$.
Hachenberger \cite{Hachenberger} constructed an element
$w\in\mathbb{Q}(\zeta_\ell)$ which is simultaneously normal in
$\mathbb{Q}(\zeta_\ell)/\mathbb{Q}(\zeta_n)$ for each
$n\in\mathcal{N}$ dividing $\ell$. The main tool is the notion of
cyclic submodules in $\mathbb{Q}(\zeta_\ell)/\mathbb{Q}(\zeta_n)$
\cite{Hachenberger2}.
\par
Let $K$ be an imaginary quadratic field other than
$\mathbb{Q}(\sqrt{-1})$ and $\mathbb{Q}(\sqrt{-3})$. For an integer
$\ell$ ($\geq2$) let $K_{(\ell)}$ be the ray class field of $K$
modulo $\ell$. Recently, Jung et tal \cite{J-K-S} showed that the
singular value of a Siegel function is normal in $K_{(\ell)}/K$. To
this end, they derived a criterion to determine a normal element in
a finite abelian extension of number fields from the Frobenius
determinant relation.
\par
Actually the criterion can be extended to determine a completely
normal element in a finite abelian extension (Theorem
\ref{criterion}). In this paper we shall give a completely normal
element in $\mathbb{Q}(\zeta_\ell)^+/\mathbb{Q}$ for an integer
$\ell$ ($\geq5$) by using the criterion (Theorems \ref{maximal} and
\ref{maximal2}). The element is expressed in terms of the cosine
function, which is simple and totally different from that of
\cite{Okada1}. Furthermore, we shall find a completely normal
element in certain extension of modular function fields in terms of
a quotient of the modular discriminant function (Theorem
\ref{modular}).

\section{A criterion for completely normal elements}

Throughout this section we let $L/K$ be a finite abelian extension
of degree $n$ ($\geq2$) with $G=\mathrm{Gal}(L/K)$. Furthermore, we
let $|\cdot|$ be a valuation  on $L$. Then $|\cdot|$ satisfies the
triangle inequality, namely, $|x+y|\leq|x|+|y|$ for all $x,y\in L$.
It follows that
\begin{equation}\label{triangle1}
|x|-|y|\leq |x+y|\quad\textrm{for all $x,y\in L$}.
\end{equation}
In particular, if $|\cdot|$ is nonarchimedean, then
$|x+y|\leq\max\{|x|,|y|\}$, from which one can readily deduce that
\begin{equation}\label{triangle2}
|x|^m-|y|^m\leq|x+y|^m\quad\textrm{for all $x,y\in L$ and any
positive real number $m$}
\end{equation}
\cite[Chapter II $\S$1]{Janusz}.

\begin{lemma}\label{character}
An element $x\in L$ is normal in $L/K$ if and only if
\begin{equation*}
\sum_{\gamma\in G}{\chi}(\gamma^{-1})x^{\gamma}\neq0
\quad\textrm{for all characters $\chi$ on $G$}.
\end{equation*}
\end{lemma}
\begin{proof}
\cite[Proposition 2.3]{J-K-S}.
\end{proof}

\begin{theorem}\label{criterion}
Assume that there exists an element $x\in L$ such that
\begin{equation*}
|x^{\gamma}/x|<1\quad\textrm{for all $\gamma\in G-\{\mathrm{Id}\}$}.
\end{equation*}
Let $m$ be any positive integer such that
\begin{equation}\label{smaller}
|x^{\gamma}/x|^m\leq 1/n\quad\textrm{for all $\gamma\in
G-\{\mathrm{Id}\}$}.
\end{equation}
Then $x^m$ is completely normal in $L/K$. In particular, if
$|\cdot|$ is nonarchimedean, then any positive power of $x$ is
completely normal in $L/K$.
\end{theorem}
\begin{proof}
Let $F$ be an intermediate field of $L/K$ with $\ell=[L:F]$ and
$H=\mathrm{Gal}(L/F)$ ($\leq G$). For any character $\chi$ on $H$ we
find that
\begin{eqnarray*}
|\sum_{\gamma\in H}{\chi}(\gamma^{-1})(x^m)^{\gamma}| &\geq&
|x^m|(1-\sum_{\gamma\in
H-\{\mathrm{Id}\}}|(x^m)^{\gamma}/x^m|)\quad\textrm{by (\ref{triangle1})}\\
&\geq&|x^m|(1-(1/n)(\ell-1))\quad\textrm{by (\ref{smaller})}\\
&=&|x^m|(n-\ell+1)/n\\&>&0\quad\textrm{because $\ell\leq n$}.
\end{eqnarray*}
This shows that $x^m$ is normal in $L/F$ by Lemma \ref{character};
and hence $x^m$ is completely normal in $L/K$. Furthermore, if
$|\cdot|$ is nonarchimedean, then we derive for any positive integer
$t$ that
\begin{eqnarray*}
|\sum_{\gamma\in H}{\chi}(\gamma^{-1})(x^t)^{\gamma}|^m &\geq&
|x^t|^m(1-\sum_{\gamma\in
H-\{\mathrm{Id}\}}|(x^t)^{\gamma}/x^t|^m)\quad\textrm{by (\ref{triangle2})}\\
&\geq&|x^t|^m(1-(1/n)^t(\ell-1))\quad\textrm{by (\ref{smaller})}\\
&\geq&|x^t|^m(1-(1/n)(\ell-1))\\
&=&|x^t|^m(n-\ell+1)/n\\&>&0\quad\textrm{because $\ell\leq n$}.
\end{eqnarray*}
Hence $x^t$ is completely normal in $L/K$ again by Lemma
\ref{character}. This completes the proof.
\end{proof}

\begin{corollary}\label{Cor}
Let $L/K$ be an abelian extension of number fields.
Assume that there exists an element $x\in L$ such that
\begin{itemize}
\item[\textup{(i)}] $x$ is an algebraic integer,
\item[\textup{(ii)}] $L=K(x)$,
\item[\textup{(iii)}] $x^\gamma$ are real for all $\gamma\in G$.
\end{itemize}
Let $a$ and $b$ be nonzero integers such that $2<|a/b|$, where $|\cdot|$
is the usual absolute value on $\mathbb{C}$. Then, a high power of $ax+b$ is completely normal in $L/K$.
\end{corollary}
\begin{proof}
Suppose that there exist distinct elements $\gamma$ and $\delta$ of $G$ such that
$|ax^\gamma+b|=|ax^\delta+b|$. Since $x^\gamma$ and $x^\delta$ are real by the assumption (iii), we get
$ax^\gamma+b=\pm(ax^\delta+b)$. Moreover, since $x^\gamma\neq x^\delta$ by the assumption (ii) and the fact
$\gamma\neq\delta$, we obtain $ax^\gamma+b=-(ax^\delta+b)$, from which it follows that
$x^\gamma+x^\delta=-2b/a$. Note that
$x^\gamma+x^\delta$ is an algebraic integer by the assumption (i), but $-2b/a$ is a rational number such that $0<|2b/a|<1$, which
yields a contradiction.
\par
For each intermediate field $F$ of $L/K$ with $[L:F]\geq2$, the
preceding argument shows that there is a unique element $\gamma_F$
of $\mathrm{Gal}(L/F)$ and a positive integer $m_F$ such that
$|(ax^\gamma+b)/(ax^{\gamma_F}+b)|^{m_F}\leq 1/[L:F]$ for all
$\gamma\in\mathrm{Gal}(L/F)-\{\gamma_F\}$. If we set
$m=\max\{m_F|F~\textrm{}\}$, then we get from Theorem
\ref{criterion} that $(ax^{\gamma_F}+b)^m$ is completely normal in
$L/F$. In particular, the set
$\{((ax^{\gamma_F}+b)^m)^\gamma|\gamma\in\mathrm{Gal}(L/F)\}$ is a
normal basis of $L$ over $F$; and hence
$((ax^{\gamma_F}+b)^m)^{\gamma_F^{-1}}= (ax+b)^m$ is normal in
$L/F$. This implies that $(ax+b)^m$ is completely normal in $L/K$
because $F$ is arbitrary. This completes the proof.
\end{proof}

\begin{remark}
If $L/K$ is an abelian extension of totally real number fields, then
there always exists such an element $x\in L$ which satisfies the assumptions of Corollary \ref{Cor}.
\end{remark}

\section{Maximal real subfields of cyclotomic fields}

Let $\ell$ be a positive integer. As is well-known,
$\mathbb{Q}(\zeta_\ell)^+=\mathbb{Q}(\zeta_\ell+\zeta_\ell^{-1})$
and
$\mathrm{Gal}(\mathbb{Q}(\zeta_\ell)^+/\mathbb{Q})\simeq(\mathbb{Z}/\ell\mathbb{Z})^\times/\{\pm1\}$,
whose actions are given as follows: if
$t\in(\mathbb{Z}/\ell\mathbb{Z})^\times/\{\pm1\}$, then
$(\zeta_\ell+\zeta_\ell^{-1})^t=\zeta_\ell^t+\zeta_\ell^{-t}$.
Denote the number of positive integers relatively prime to $\ell$ by
$\phi(\ell)$. Then we have
\begin{equation*}
[\mathbb{Q}(\zeta_\ell)^+:\mathbb{Q}]=\left\{\begin{array}{ll}1, &
\textrm{if}~\ell=1,2,3,4,6,\\
\phi(\ell)/2~(\geq2),&\textrm{otherwise}
\end{array}\right.
\end{equation*}
\cite[Chapter 2]{Washington}.
\par
Let $|\cdot|$ denote the usual absolute value on $\mathbb{C}$.

\begin{theorem}\label{maximal}
Let $\ell$ \textup{($\neq1,2,3,4,6$)} be a positive integer. If $m$
is any positive integer such that
\begin{equation*}
((\cos(4\pi/\ell)+1)/(\cos(2\pi/\ell)+1))^m\leq2/\phi(\ell),
\end{equation*}
then $(\cos(2\pi/\ell)+1)^m$ is completely normal in
$\mathbb{Q}(\zeta_\ell)^+/\mathbb{Q}$.
\end{theorem}
\begin{proof}
Let $x=(\zeta_\ell+\zeta_\ell^{-1})/2+1=\cos(2\pi/\ell)+1$. If
$\gamma\in
\mathrm{Gal}(\mathbb{Q}(\zeta_\ell)^+/\mathbb{Q})-\{\mathrm{Id}\}$,
then $x^\gamma=(\zeta^t+\zeta^{-t})/2+1$ for some integer $t$ with
$\gcd(\ell,t)=1$ and $1<t\leq[\ell/2]$, where $[\cdot]$ is the Gauss
symbol. We achieve that
\begin{eqnarray*}
|x^\gamma/x|&=&|((\zeta^t+\zeta^{-t})/2+1)/((\zeta+\zeta^{-1})/2+1)|\\
&=&|(\cos(2t\pi/\ell)+1)/(\cos(2\pi/\ell)+1)|\\
&\leq&|(\cos(4\pi/\ell)+1)/(\cos(2\pi/\ell)+1)|,
\end{eqnarray*}
which is less than $1$. The result follows from Theorem
\ref{criterion}.
\end{proof}

\begin{theorem}\label{maximal2}
Let $\ell$ \textup{($\geq5$)} be an odd integer. If $m$ is any
positive integer such that
\begin{equation*}
(\cos(2\pi/\ell)/\cos(\pi/\ell))^m\leq 2/\phi(\ell),
\end{equation*}
then $\cos^m(\pi/\ell)$ is completely normal in
$\mathbb{Q}(\zeta_\ell)^+/\mathbb{Q}$.
\end{theorem}
\begin{proof}
Let $x=-(\zeta^{(\ell-1)/2}+\zeta^{-(\ell-1)/2})/2$. Since
$1\cdot\ell+(-2)\cdot(\ell-1)/2=1$, we get
$\gcd(\ell,(\ell-1)/2)=1$, which implies that $x$ is a conjugate of
$-(\zeta+\zeta^{-1})/2$. Hence, if
$\gamma\in\mathrm{Gal}(\mathbb{Q}(\zeta_\ell)^+/\mathbb{Q})-\{\mathrm{Id}\}$,
then $x^\gamma=-(\zeta^t+\zeta^{-t})/2$ for some integer $t$ with
$\gcd(\ell,t)=1$ and $1\leq t<(\ell-1)/2$. We find that
\begin{eqnarray*}
|x^\gamma/x|&=&|(-(\zeta^t+\zeta^{-t})/2)/(-(\zeta^{(\ell-1)/2}+\zeta^{-(\ell-1)/2})/2)|\\
&=&|-\cos(2t\pi/\ell)/-\cos(\pi-\pi/\ell)|\\
&=&|\cos(2t\pi/\ell)|/\cos(\pi/\ell)\\
&\leq&\cos(2\pi/\ell)/\cos(\pi/\ell),
\end{eqnarray*}
which is less than $1$. We obtain the assertion by Theorem
\ref{criterion}.
\end{proof}

\begin{lemma}\label{composite}
Let $L_1$ and $L_2$ be finite Galois extensions of a number field
$K$ such that $L_1\cap L_2=K$. If $x_k\in L_k$ is normal in $L_k/K$
\textup{($k=1,2$)}, then $x_1x_2$ is normal in $L_1L_2/K$.
\end{lemma}
\begin{proof}
\cite[p.227]{Kawamoto}.
\end{proof}

\begin{lemma}\label{quadratic}
Let $t=4$ or an odd prime $p$ such that $p\equiv3\pmod{4}$. Then
$\mathbb{Q}(\zeta_t)$ contains a unique quadratic extension of
$\mathbb{Q}$, namely $\mathbb{Q}(\sqrt{-t})$.
\end{lemma}
\begin{proof}
\cite[Theorem 11.1]{Janusz}.
\end{proof}

\begin{theorem}
Let $t=4$ or an odd prime $p$ such that $p\equiv3\pmod{4}$. Let
$\ell$ \textup{($\neq1,2,3,4,6$)} be a positive integer. If $m$ is
any positive integer such that
\begin{equation}\label{condition}
((\cos(4\pi/t\ell)+1)/(\cos(2\pi/t\ell)+1))^m\leq 2/\phi(t\ell),
\end{equation}
then $(\sqrt{-t}+1)(\cos(2\pi/t\ell)+1)^m$ is normal in
$\mathbb{Q}(\zeta_{t\ell})/\mathbb{Q}$.
\end{theorem}
\begin{proof}
One can readily show that $\sqrt{-t}+1$ is normal in
$\mathbb{Q}(\sqrt{-t})/\mathbb{Q}$. And, if $m$ is any positive
integer which satisfies the condition (\ref{condition}), then
$(\cos(2\pi/t\ell)+1)^m$ is normal in
$\mathbb{Q}(\zeta_{t\ell})^+/\mathbb{Q}$ by Theorem \ref{maximal}.
On the other hand, since $\mathbb{Q}(\sqrt{-t})$ is an imaginary
quadratic field contained in $\mathbb{Q}(\zeta_t)$
($\subset\mathbb{Q}(\zeta_{t\ell})$) by Lemma \ref{quadratic}, we
have
$\mathbb{Q}(\sqrt{-t})\cap\mathbb{Q}(\zeta_{t\ell})^+=\mathbb{Q}$
and
$\mathbb{Q}(\sqrt{-t})\mathbb{Q}(\zeta_{t\ell})^+=\mathbb{Q}(\zeta_{t\ell})$.
Now, the result follows from Lemma \ref{composite}.
\end{proof}

\section {Fields of modular functions}

Let $\mathbb{H}=\{\tau\in\mathbb{C}~|~\mathrm{Im}(\tau)>0\}$ be the
complex upper half-plane. For a positive integer $N$ we consider the
group
\begin{equation*}
\Gamma_0(N)=\bigg\{\begin{pmatrix}a&b\\c&d\end{pmatrix}\in\mathrm{SL}_2(\mathbb{Z})~|~
c\equiv0\pmod{N}\bigg\},
\end{equation*}
which acts on $\mathbb{H}^*=\mathbb{H}\cup\mathbb{Q}\cup\{\infty\}$
as fractional linear transformations. Then a (meromorphic)
\textit{modular function} for $\Gamma_0(m)$ is a $\mathbb{C}$-valued
function on $\mathbb{H}$, except for isolated singularities, which
satisfies the following three conditions:
\begin{itemize}
\item[(i)] $f(\tau)$ is meromorphic on $\mathbb{H}$,
\item[(ii)] $f(\tau)$ is invariant under $\Gamma_0(N)$,
\item[(iii)] $f(\tau)$ is meromorphic at the cusps
$\mathbb{Q}\cup\{\infty\}$
\end{itemize} \cite[$\S$11 B]{Cox}.
We denote the field of all modular functions for $\Gamma_0(N)$ by
$\mathbb{C}(X_0(N))$. As is well-known, $\mathbb{C}(X_0(N))$ is a
Galois extension of $\mathbb{C}(X_0(1))$ whose Galois group is
isomorphic to the quotient group $\Gamma_0(1)/\Gamma_0(N)$
\cite[Chapter 6]{Lang}.
\par
For a pair $(r_1,r_2)\in\mathbb{Q}^2-\mathbb{Z}^2$, the
\textit{Siegel function} $g_{(r_1,r_2)}(\tau)$ on $\mathbb{H}$ is
defined  by the following infinite product
\begin{eqnarray*}
g_{(r_1,r_2)}(\tau)=-q^{(1/2)\textbf{B}_2(r_1)}e^{\pi
ir_2(r_1-1)}(1-q^{r_1}e^{2\pi
ir_2})\prod_{n=1}^{\infty}(1-q^{n+r_1}e^{2\pi
ir_2})(1-q^{n-r_1}e^{-2\pi ir_2}),
\end{eqnarray*}
where $q=e^{2\pi i\tau}$ and $\textbf{B}_2(X)=X^2-X+1/6$ is the
second Bernoulli polynomial.
\par
For $X\in\mathbb{R}$ we let $\langle X\rangle$ be the fractional
part of $X$ in the interval $[0,1)$.

\begin{lemma}\label{transformation}
Let $N$ \textup{($\geq 2$)} be an integer and
$(r_1,r_2)\in(1/N)\mathbb{Z}^2-\mathbb{Z}^2$.
\begin{itemize}
\item[\textup{(i)}] $g_{(r_1,r_2)}(\tau)^{12N}$ is determined by
$\pm(r_1,r_2)\pmod{\mathbb{Z}^2}$.
\item[\textup{(ii)}] If
$\left(\begin{smallmatrix}a&b\\c&d\end{smallmatrix}\right)\in\mathrm{SL}_2(\mathbb{Z})$,
then
\begin{equation*}
g_{(r_1,r_2)}(\tau)^{12N}\circ\alpha
=g_{(r_1,r_2)\alpha}(\tau)^{12N}=
g_{(r_1a+r_2c,r_1b+r_2d)}(\tau)^{12N}.
\end{equation*}
\item[\textup{(iii)}] $\mathrm{ord}_q~g_{(r_1,r_2)}(\tau)=(1/2)\mathbf{B}_2(\langle
r_1\rangle)$.
\end{itemize}
\end{lemma}
\begin{proof}
\cite[Chapter 2 $\S$1]{K-L}.
\end{proof}

Let
\begin{equation*}
\Delta(\tau)=(2\pi)^{12}q\prod_{n=1}^\infty(1-q^n)^{24}\quad(\tau\in\mathbb{H})
\end{equation*}
be the \textit{modular discriminant function}.

\begin{lemma}\label{product}
We have the relation
\begin{equation*}
\Delta(\tau)/\Delta(N\tau)=N^{12}\prod_{k=1}^Ng_{(0,k/N)}(\tau)^{-12},
\end{equation*}
which is a modular function for $\Gamma_0(N)$.
\end{lemma}
\begin{proof}
\cite[Proposition 5.1]{K-S}.
\end{proof}

\begin{theorem}\label{modular}
Let $N$ \textup{($\geq2$)} be an integer. Let $L=\mathbb{C}(X_0(N))$
and $K$ be the subfield of $L$ fixed by
$\left(\begin{smallmatrix}1&0\\1&1\end{smallmatrix}\right)$ in
$\Gamma_0(1)/\Gamma_0(N)$. Then, any positive power of
$\Delta(\tau)/\Delta(N\tau)$ is completely normal in $L/K$.
\end{theorem}
\begin{proof}
By Galois theory we have
\begin{eqnarray*}
\mathrm{Gal}(L/K)&\simeq&\bigg\langle
\begin{pmatrix}
1&0\\1&1
\end{pmatrix} \bigg\rangle~\textrm{in}~\Gamma_0(1)/\Gamma_0(N)\\
&=&\bigg\{\begin{pmatrix}1&0\\t&1\end{pmatrix}~|~t=0,1,\cdots,
N-1\bigg\}~\textrm{in}~\Gamma_0(1)/\Gamma_0(N).
\end{eqnarray*}
Consider the nonarchimedean valuation $|\cdot|$ on $L$ defined by
\begin{eqnarray*}
|\cdot|~:~L&\longrightarrow&\mathbb{R}_{\geq0}\\
\alpha&\mapsto& |\alpha|=\exp(-\mathrm{ord}_{q}~\alpha).
\end{eqnarray*}
Let $x=\Delta(\tau)/\Delta(N\tau)$. For any
$\gamma=\left(\begin{smallmatrix}1&0\\t&1\end{smallmatrix}\right)\in
\mathrm{Gal}(L/K)-\{\mathrm{Id}\}$ we find that
\begin{eqnarray*}
|x^\gamma/x|^N&=&\bigg|(N^{12N}\prod_{k=1}^{N-1}g_{(0,k/N)}(\tau)^{-12N})^\gamma
/N^{12N}\prod_{k=1}^{N-1}g_{(0,k/N)}(\tau)^{-12N}\bigg|\quad\textrm{by Lemma \ref{product}}\\
&=&\bigg|\prod_{k=1}^{N-1}g_{(kt/N,k/N)}(\tau)^{-12N}/\prod_{k=1}^{N-1}g_{(0,k/N)}(\tau)^{-12N}\bigg|
\quad\textrm{by Lemma \ref{transformation}(i) and (ii)}\\&=&
\exp\bigg(-\sum_{k=1}^{N-1}(1/2)\mathbf{B}_2(\langle
kt/N\rangle)\cdot(-12N)+\sum_{k=1}^{N-1}(1/2)\mathbf{B}_2(0)\cdot(-12N)\bigg)\\
&&\textrm{by Lemma \ref{transformation}(iii)}\\
&=&\exp\bigg(6N\sum_{k=1}^{N-1}(\mathbf{B}_2(\langle
kt/N\rangle)-\mathbf{B}_2(0))\bigg)\\
&<&1\quad\textrm{because $\mathbf{B}_2(X)$ has its maximum at $X=0$
in the interval $[0,1)$},
\end{eqnarray*}
from which it follows that $|x^\gamma/x|<1$. Therefore $x$ is
completely normal in $L/K$ by Theorem \ref{criterion}.
\end{proof}

\bibliographystyle{amsplain}

\address{
Department of Mathematical Sciences \\
KAIST \\
Daejeon 373-1 \\
Korea 305-701} {jkkoo@math.kaist.ac.kr}
\address{
Department of Mathematics and Statistics \\
McGill University \\
805 Sherbrooke St. West\\
Montreal, Quebec\\
Canada H3A 2K6  } {dong.hwa.shin@mail.mcgill.ca}

\end{document}